\documentclass{article}
\usepackage{amssymb}
\usepackage{amsthm}
\usepackage{autonum}
\usepackage{graphicx}
\usepackage{comment}
\usepackage{wrapfig}
\usepackage{placeins}
 \usepackage{url} 
\newtheorem{theorem}{Theorem}

\newtheorem{lemma}[theorem]{Lemma}

\newtheorem{definition}[theorem]{Definition}
\newtheorem{example}[theorem]{Example}

\newcommand{\R}{{\mathbb R}}

\DeclareMathOperator{\AC}{AC}

\begin{document}

\title{The Lavrentiev Phenomenon}
\author{
Rapha\"el  Cerf
and
 Carlo Mariconda
}
\maketitle

\begin{abstract} The basic problem of the calculus of variations consists of finding a function that minimizes an energy,
	like finding the fastest trajectory between two points
	for a point mass in a gravity field
	or finding the best shape of a wing. The existence of a solution may
	be established in quite abstract spaces, much larger than the space of smooth functions.
	An important practical problem is
	that of being able to approach the value of the
	infimum of the energy.
	However, numerical methods work with very ``concrete'' functions and sometimes they are unable to approximate the infimum:
this is the surprising Lavrentiev phenomenon. The papers that ensure the non-occurrence of the phenomenon form a recent saga, and the most general result formulated in the early '90s was actually fully proved just recently, more than 30 years later.
Our aim here
is to introduce the reader to the calculus of variations,
to illustrate the Lavrentiev phenomenon with the simplest
known example, and to give an elementary proof of the non-occurrence of the phenomenon.
\end{abstract}

\noindent

\section{Introduction.}
Consider a positive smooth function $y:[a,b]\to [0, +\infty)$. The area of the surface obtained by rotating the graph of $t\mapsto y(t)$ around the $t$-axis  (see Figure~\ref{fig:rotation_surface}) is given by
\begin{equation}\label{ex1}
F(y)=2\pi\int_a^by(t)\sqrt{1+y'^2(t)}\,dt.\end{equation}
Once we fix the initial and final values $y(a)=A, y(b)=B$, is there a function that gives the infimum of $F$?
It turns out that the answer is positive if the values of $A, B$ are big enough with respect to $b-a$, otherwise
the infimum is
the sum of the area of the circles of radii $A$ and $B$, but of course there is no function whose graph is the union of the segments $[(a,A), (a,0)], [(a,0), (b,0)], [(b,0), (b,B)]$, and the
infimum of $F$ is thus not reached.
Problems like this fall into a general scheme,
called the {Basic problem of the Calculus of Variations},
which we describe next.
\begin{figure}[!ht]
\begin{center}
\vskip-10pt
\includegraphics[width=0.7\textwidth]{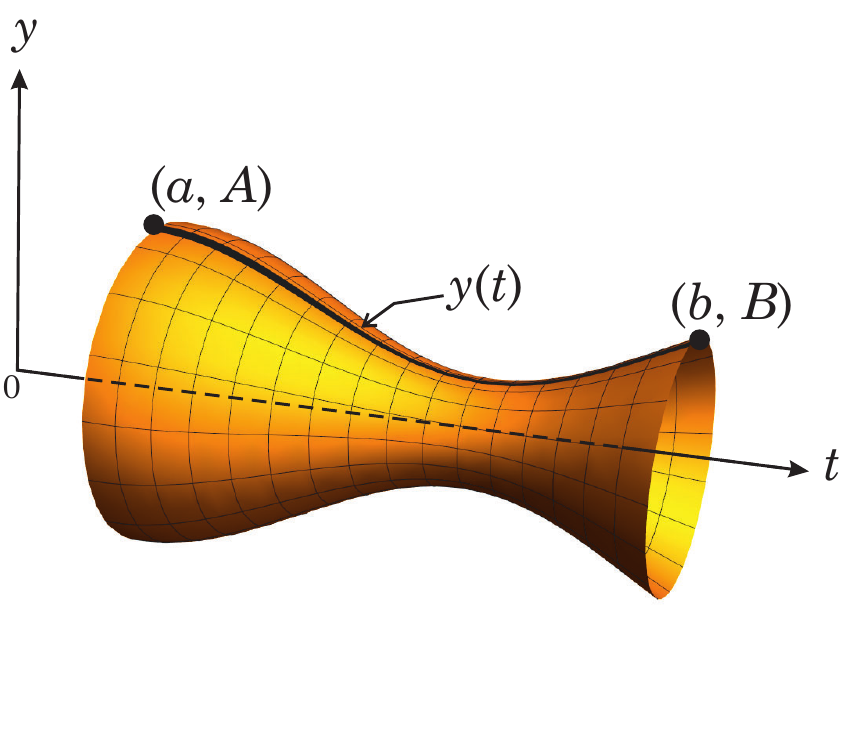}
\vskip-30pt
\caption{
{A surface obtained by rotating the curve $y(t)$ around the $t$--axis.}}
\label{fig:rotation_surface}
\end{center}
\vskip-20pt
\end{figure}

\noindent
Let $L(t, y,v)$ be a non--negative $C^1$ function defined on
$[a,b]\times\R\times\R$.
To every continuous, piecewise $C^1$ function  $y:[a,b]\to \R$ (or just absolutely continuous, see below), we associate the \textbf{ integral functional}
\begin{equation}\label{if}
F(y)=\int_a^bL(t, y(t), y'(t))\,dt.\end{equation}
The Basic Problem of the Calculus of Variations, briefly denoted by
\begin{equation}\label{tag:P}\min\, F(y):\quad y\in\mathcal E,\quad  y(a)=A,\, y(b)=B\,,\end{equation}
is to find a function $y$ that minimizes the value $F(y)$ among the functions $y$ belonging to a suitable space $\mathcal E$   and satisfying the \textbf{boundary conditions} $y(a)=A, y(b)=B$.

In this context $L$ is called a \textbf{Lagrangian} in honour of Joseph-Louis Lagrange, who was the first to prove the Euler-Lagrange equation (only conjectured by Leonard Euler); actually the name Calculus of Variations arises from the variation technique of Lagrange's proof.
When dealing with $L$ as a function of two variables, we use the letter $y$ for the position and $v$ for the speed, so we write $L(t, y,v)$. Inside the integral functional $F(y)$, the letter $v$ is naturally
replaced by the time derivative
$y'(t)$ of the trajectory.

The  first  question that has to be addressed is whether there
exists a function $y$ in $\mathcal E$ realizing the infimum of $F$,
that is whether \eqref{tag:P} has a solution. Of course, the answer to this
question depends crucially on
the choice of
the admissible class $\mathcal E$ of functions.
A natural possible choice  for $\mathcal E$  is the
set of the functions which are of class $C^1$ on $[a,b]$, or
even continuous and piecewise $C^1$. Unfortunately, there is no satisfactory general
existence result for this choice. The most natural choice is to work with
absolutely continuous functions, whose definition is recalled next.
\begin{definition}
	\label{absco}
	A function $y$ belongs to the space
	$\AC([a,b])$
	of the \textbf{absolutely continuous}
	functions on the interval $[a,b]$
	if and only if there exists a Lebesgue integrable function $f$
	defined on $[a,b]$ such that
	\begin{equation}\label{irep}
	\forall t\in [a,b]\qquad y(t)=y(a)+\int_a^tf(s)\,ds\,.\end{equation}
\end{definition}
\noindent
A function $y$ admitting a representation like~\eqref{irep} is differentiable
almost everywhere, its derivative $y'$ being equal to $f$ outside a Lebesgue
negligible set. In particular, the integral functional $F(y)$ of
such a function is well-defined and the problem \eqref{tag:P} can be studied
within the class $\mathcal E= \AC([a,b])$.
Moreover the space
$\AC([a,b])$
contains not only the continuous and piecewise $C^1$ functions but also the
Lipschitz functions on $[a,b]$
(but it is far from obvious to prove
that a Lipschitz function admits a representation like~\eqref{irep}).
In fact,
Lipschitz functions on $[a,b]$ can be identified as the
functions of
$\AC([a,b])$ whose
derivative (defined outside a Lebesgue negligible set)
is bounded on $[a,b]$.
For instance, the square root function $\smash{\sqrt t}$ is absolutely continuous
but not Lipschitz on $[0,1]$ since $y'(t)=\smash{1/({2\sqrt t})}$,
though integrable,  is not bounded on $[0,1]$.
Functionals $F$ whose minima belong to the Lipschitz functions have the advantage that they can be handled with numerical methods.
However, there are cases in which the minima of $F$
exist but are not Lipschitz.
\noindent
Leonida Tonelli's existence result (see \cite{GBH}) guarantees the existence of
a solution belonging to the space
$\mathcal E =\AC([a,b])$,
once
some reasonable conditions are satisfied.

In any case,
whether a minimizer of $F$ exists or not,
we would like to approach the value of the infimum of $F$ (which exists since $F\ge 0$)
with very concrete functions. To this end,
we implement numerical
methods, which traditionally involve functions
with bounded derivatives. In some cases, this works very well, as shown in the following example.
\begin{example}\label{ex:yy}
Consider the problem of minimizing
\begin{equation}\label{ex3}
F(y)=\int_0^1(2yy'-1)^2\,dt,\quad y(0)=0,\, y(1)=1\,.\end{equation}
The function
$y_*(t)=\sqrt{t}$ is the global minimum of $F$.
Yet its derivative
$y'_*(t)$ is unbounded.
Is there a way to find a sequence $(y_n)_{n\geq 1}$ of piecewise $C^1$ functions with bounded derivatives with $y_n(0)=0, y_n(1)=1$, and such that $F(y_n)\to F(y_*)=0$ as $n\to +\infty$? The answer is affirmative here.
Consider the sequence
$(y_n)_{n\geq 1}$
depicted in Figure~\ref{fig:smoothy} and defined by
$$y_n(t)=\begin{cases}\,tn^{1/2}&t\in[0, 1/n]\\
\,t^{1/2}&t\in[1/n,1]\end{cases}\,.$$
\begin{figure}[!ht]
\begin{center}
\includegraphics[width=0.8\textwidth]{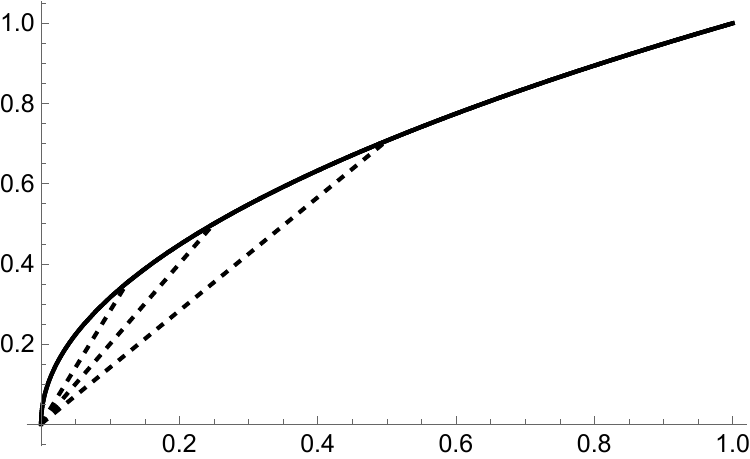}
\caption{{\small The sequence $y_n$ for Example~\ref{ex:yy}.}}
\label{fig:smoothy}
\end{center}
\vskip-10pt
\end{figure}
\FloatBarrier
\noindent
The function
$y_n$ has a bounded derivative and
\[F(y_n)\,=\,\int_0^{1/n}\kern-7pt
	(2tn-1)^2\,dt\,\le\,
\int_0^{1/n}\kern-7pt
\big((2tn)^2+1\big)\,dt\,\leq\,\dfrac3n\to 0\quad\text{as}\quad n\to +\infty\,.\]
\end{example}
In the general case this approximation is not obvious at all: this depends partly on the fact that the integral is not continuous with respect to pointwise convergence of functions.
Mikhail Lavrentiev in \cite{Lavrentiev} discovered in 1926 that there are innocent-like  problems of the Calculus of Variations for which there is a
strictly positive gap between
the infimum value of $F$ and the value taken at any function with bounded derivative.
Basilio Mani\`a in \cite{Mania}  gave a simplified example (presented in Example~\ref{ex:mania2}), where
\begin{equation}\label{ex4}
L(t,y,v)=(y^3-t)^2v^6,\quad y(0)=0,\, y(1)=1\,.\end{equation}
Notice that this Lagrangian is non--autonomous, meaning that it depends on the
time variable~$t$.
What happens here is that there is a minimizer $y_*$ and $F(y_*)=0$; however there is $\eta>0$ such that $F(y)\ge\eta$ whenever $y$ is Lipschitz and satisfies the boundary conditions.
The story of the conditions that ensure  the non-occurrence of the phenomenon is a sort of mathematical saga.
Of course there is no phenomenon if one knows that the minimizer exists and is Lipschitz, i.e., has a bounded derivative. This direction is part of the so-called  regularity theory that began with the
existence theory by Tonelli himself in early 1900s and had a new impulse in the 80s with the pioneering work by Francis Clarke and Richard Vinter \cite{CVTrans}: all results in this framework require a kind of nonlinear growth condition from below on $L$. Otherwise,  unless one assumes unnatural growth conditions of $L$ from above,
the  technical problem is to pass to the limit under the integral sign.
In Mani\`a's example, the Lagrangian depends on the time variable $t$. A celebrated result by Giovanni Alberti and Francesco Serra Cassano in \cite{ASC} shows that if $L$ is autonomous
(i.e., $L$ is as in problem~$\eqref{tag:P}$, it does not depend on $t$) and just {\em bounded on bounded sets}, then the phenomenon does not occur. Actually, the proof is given there just for problems with one prescribed boundary condition $y(a)=A$. The case with two prescribed boundary data is not only technically different: indeed  there are  cases, like the one described in  Problem~\eqref{mania2}, where the phenomenon may occur if one fixes two end-points but no more if one let one end-point to vary. The  general case was conjectured by Giovanni Alberti and proven by Carlo Mariconda in \cite{CM5}, more than 30 years later.
The methods around the Lavrentiev phenomenon are almost elementary, but as the previous
short story shows, the intuition is fallacious and there are steps that always must be carefully justified.

The aim of this paper is twofold:
we give the simplest known
example exhibiting the Lavrentiev phenomenon (which is a
further simplification of an example given in
\cite{CerfM})
and we provide a very elementary proof of the avoidance of the phenomenon, under some extra assumptions on the Lagrangian $L(y,v)$, namely  convexity in $v$. It is inspired by the one given in \cite{CFM} by  Arrigo Cellina, Alessandro Ferriero and Elsa Marchini for continuous Lagrangians,
and to its recent nonsmooth extension provided in \cite{CM7, CM8}. We point out that \cite{CFM} was the first paper with a complete proof of the non occurrence of the Lavrentiev phenomenon for the problem with {\em two} prescribed end point conditions,  without the need of any kind of growth conditions.
Before embarking into this program, we provide a little introduction to the Calculus of Variations through the presentation of some classical examples,
which illustrate and raise the basic
questions of the existence and regularity of minimizers.
We present the two fundamental conditions satisfied by the minima of the problem \eqref{tag:P}, namely the equations of
Euler--Lagrange and Du Bois-Reymond. We finally give the proof of
the avoidance of the phenomenon. This proof is quite long, it is divided into seven steps that involve some standard results from measure theory, a judicious reparametrization of
an approximating solution and a delicate control of its energy.

The Lavrentiev phenomenon is observed in more general contexts; however, our focus is solely on the one-dimensional case. We recommend that readers who are well-versed in this subject consult \cite{BB} for a thorough survey and  \cite{BMT, BousquetAUTON} for the latest findings in the multidimensional, autonomous scenario.
\section{Classical examples and questions.}
Examples and concrete problems have played a key role
in the development of the theory
of Calculus of Variations. In this section, we present some classical examples
and important questions raised by them.
At the beginning of the introduction, we mentioned the problem of
minimal rotation surfaces, that consists of minimizing \eqref{ex1}. Here is another classical question.
\begin{example}[\textbf{Brachistocrone}]\label{ex:Brachi}
Find the quickest
	trajectory
	of a point mass
in a gravity field
	between two points $P=(a, A)$ and $Q=(b, B)$.
	The problem was formulated by Galileo Galilei in 1638 who conjectured, erroneously, that the solution was a circular path. The correct solution was found by Johann Bernoulli in 1697.
	In a Cartesian coordinate system with gravity acting in the direction of the negative $y$ axis  and $A>B$,  the problem consists in
	finding $y:[a,b]\to \mathbb R$ that solves the problem
	\begin{equation}\label{ex5}
		\min\int_a^b\dfrac{\sqrt{1+y'^2}}{\sqrt{A-y}}\,dt
:\, y(a)=A,\, y(b)=B\,.
\end{equation}
In Figure~\ref{fig:brachi} is the solution, it is a cycloid arc called the
brachistochrone curve (in ancient greek, brachistochrone means 'shortest time').
\begin{figure}[!ht]
\begin{center}
\includegraphics[width=0.7\textwidth]{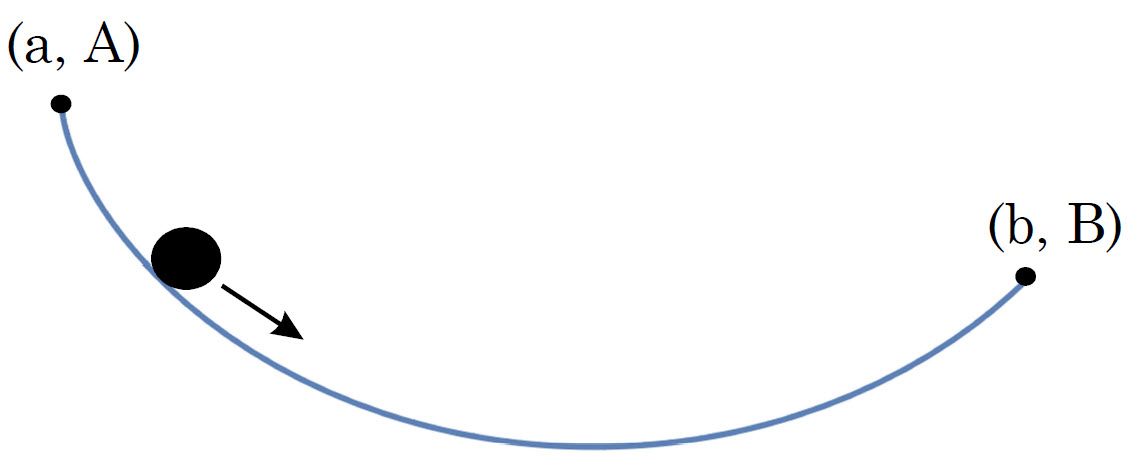}
\caption{{\small The brachistocrone in Example~\ref{ex:Brachi}.}}
\label{fig:brachi}
\end{center}
\vskip-10pt
\end{figure}
\end{example}
\noindent
In particular, the solution is $C^1$, so this problem can be handled within
the space $\mathcal E=C^1$.
Unfortunately,
as the next example demonstrates,
the minimum of $F$ may exist out of the
class of $C^1$ functions.  From now onwards,
we denote by $I$ the interval of definition of the
unknown function.
\begin{example}
[\textbf{A simple example with no $C^1$ solution}]\label{ex:noxistenceC1}
The problem
\begin{equation}\label{ex7}
\min \,F(y)=\int_0^1(y'^2-1)^2\,dt:\, y(0)=0, \, y(1)=0\,,\end{equation}
has no solutions among $C^1$ functions. Indeed, let
$$y_*(t)=\begin{cases}t&\text{ if }t\in [0, 1/2],\\ 1-t&\text{ if }t\in[1/2,1].\end{cases}\,$$
Notice that $(y'_*)^2=1$.
For each $n\geq 1$, one may build a $C^1$ function $y_n$
such that:
\begin{itemize}
\item  $0\le y_n\le y_*$ on $[0,1]$;
\item $y_n(t)=y_*(t)$ for all $t\in I\setminus [1/2-1/n, 1/2+1/n]$;
\item $0\le y'_n\le 1$ on $[0,1]$\,.
\end{itemize}
\begin{figure}[!ht]
\begin{center}
\vskip-10pt
\includegraphics[width=0.7\textwidth]{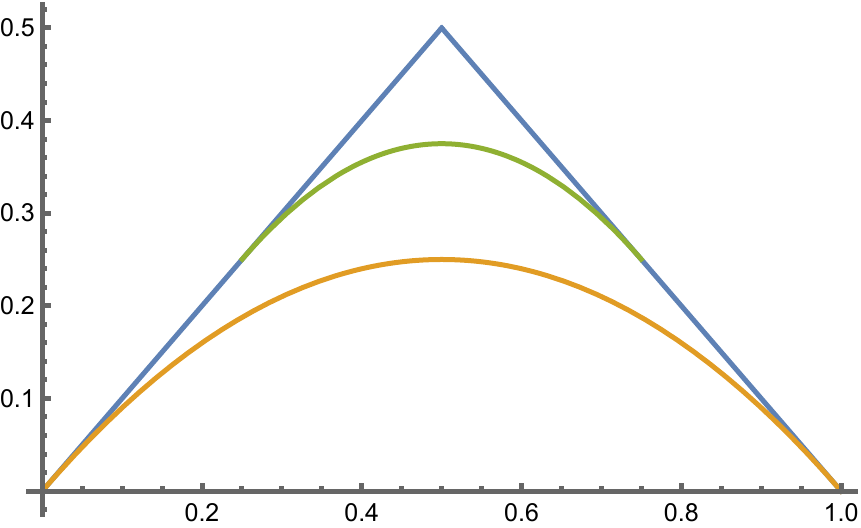}
\caption{The functions $y_n$ in Example~\ref{ex:noxistenceC1} have an energy converging to $0$, but their limit is not $C^1$.}
\label{fig:smooth}
\end{center}
\vskip-10pt
\end{figure}
See Figure~\ref{fig:smooth} for the graph of these functions.
Thus \[0\le F(y_n)=\displaystyle\int_{\frac12-\frac1n}^{ \frac12+\frac1n}(y_n'^2-1)^2\,dt\le \int_{\frac12-\frac1n}^{ \frac12+\frac1n}1\,dt=\dfrac2n\to 0\quad
\text{as}\quad n\to +\infty,\]
showing that $0$ is the infimum  of $F$ among the $C^1$ functions with the prescribed boundary conditions. Nevertheless, there is no $C^1$ function $y$ with $y(0)=y(1)=0$ such that $F(y)=0$, otherwise $y'^2=1$ on $I$ so that either $y'=1$ on $I$ (implying $y(1)=1$) or $y'=-1$ on $I$ (implying $y(1)=-1$).
However, the function $y_*$ is   a minimizer of $F$ among Lipschitz functions.
\end{example}
This example shows that
we need to enlarge the class of $C^1$ functions.
A natural choice would be to work with piecewise $C^1$ functions.
The previous example, which has no $C^1$ solution,
can be perfectly handled within this class.
This is not the case for the next example.
\begin{example}[\textbf{Mani\`a's example \cite{Mania}}]\label{ex:mania}
Consider the problem
\begin{equation}
\label{tag:mania}
\min F(y)=\int_0^1(y^3-t)^2y'^6\,dt:\quad y\in\AC(I),\, y(0)=0,\, y(1)=1\,.\end{equation}
The Lagrangian here is non--autonomous because
it involves the time variable $t$.
The function $y_*(t)=t^{1/3}$
satisfies the boundary conditions
and $F(y_*)=0$.
Its derivative,
$y_*'(t)\,=\,t^{-2/3}/3$,
is unbounded but nevertheless integrable on $[0,1]$.
Thus $y_*$ is the minimizer of $F$ among the absolutely continuous functions. Notice that if $F(y)=0$ then, necessarily, $y=y_*$, thus $F$ has no Lipschitz minimizers.
\end{example}
In any case, once a class of admissible functions $\mathcal E$ is chosen, the infimum $\inf_{\mathcal E}F$ of the integral
functional $F$ among the functions of $\mathcal E$ exists for sure, since $F$ is bounded from below by $0$.
A \textbf{minimizing sequence} for \eqref{tag:P}  is  a sequence $(y_n)_{n\geq 1}$ of functions in $\mathcal E$ with $y_n(a)=A, y_n(b)=B$ and such that $F(y_n)\to \inf_{\mathcal E} F$ as $n\to +\infty$. We say that \eqref{tag:P} has a minimizer or a solution
whenever there is $y_*$ in $\mathcal E$ with $y_*(a)=A, y_*(b)=B$ such that $F(y_*)=\inf_{\mathcal E}F$.
Two main problems of the Calculus of Variations are:
\smallskip

\noindent
{\bf Existence.} Does problem \eqref{tag:P} admit a solution in a suitable function space $\mathcal E$, possibly larger than $C^1$ functions, e.g., Lipschitz or absolutely continuous functions?
\smallskip

\noindent
{\bf Regularity.} Once a solution to \eqref{tag:P} exists, can one show that it actually belongs to a space of ``more'' regular functions?
\smallskip

\noindent
It turns out that a satisfactory class for finding minimizers are
the absolutely continuous functions $\AC(I)$
(see definition~\ref{absco}):
a celebrated result by Tonelli
(see, for instance, \cite[\S 3.2]{GBH})
establishes some sufficient criteria under which a minimizer of $F$ exists in that class. Thus, from now onwards, we
will work with $\mathcal E=\AC(I)$.
Concerning the regularity problem, many results,
starting from \cite{CVTrans}, give some conditions ensuring that the minimizer, whenever it exists,  is actually Lipschitz.
It may happen however that \eqref{tag:P} has no solutions, even among the absolutely continuous functions.
\begin{example}[An example with no solution in $\AC(I)$]\label{ex:sawtooth}
Consider the problem
\begin{equation}
\label{tag:sawtooth}
F(y)=\int_0^1\big((y'^2-1)^2\,+y^2\big)\,dt: \quad y(0)=y(1)=0\,.\end{equation}
For  $n\ge 2$, let $y_n$ be the 'sawtooth' function (see Figure \ref{fig:sawtooth}) defined by
\[y_n(0)=0,\quad y_n'(t)=\begin{cases}1&\text{ if } t\in\big[\frac{2k}{2n},\frac{2k+1}{2n}[\\
-1&\text{ if } t\in\big[\frac{2k+1}{2n},\frac{2k+2}{2n}[\end{cases}, \quad k=0,...,n-1\,.\]
Then $y_n'\in\{-1,1\}$ and $|y_n|\le 1/({2n})$; therefore
\[0\le F(y_n)=\int_0^1y_n^2\,dt\le \dfrac1{(2n)^2}\to 0
\quad\text{as}\quad n\to +\infty,\]
showing that $\inf F=0$ among the absolutely continuous functions with the given constraints.
However if $y$ is absolutely continuous and $F(y)=0$,
then necessarily $\smash{\int_0^1y^2(t)}\,dt=0$ implying that $y=0$,
so that $F(y)=\smash{\int_0^11^2\,dt}=1$,  which is absurd.
\begin{figure}[!ht]
\begin{center}
\includegraphics[width=0.85\textwidth]{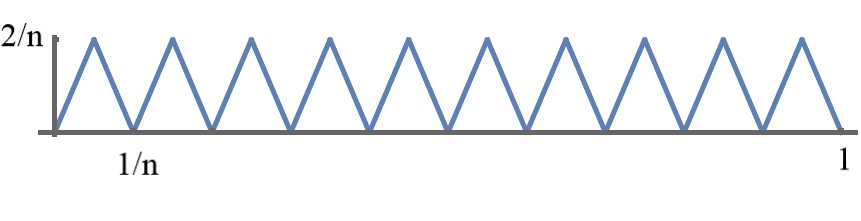}
\vskip-10pt
\caption{The 
functions $y_n$ for Example \ref{ex:sawtooth} have an energy converging to $0$, but their limit does not exist.}
\label{fig:sawtooth}
\end{center}
\vskip-20pt
\end{figure}
\end{example}
\section{Euler--Lagrange and Du Bois-Reymond equations.}
Like Fermat's rule for the derivative of a
function of one real variable, necessary conditions are helpful to find the candidates for being a minimum of \eqref{tag:P}.
 We denote by $L_t, L_y, L_v$ the partial derivatives of $L$ with respect to the  variables $t, y, v$.
\begin{theorem}[Necessary conditions]
\label{th:euler}
Assume that $y_*$ is an absolutely continuous minimizer of \eqref{tag:P}. Then
it satisfies the
\textbf{Euler--Lagrange equation}:
for almost every $t\in I$,
\begin{equation}
\label{eula}
L_y(t, y_*(t), y_*'(t))-\frac{d}{dt}L_v(t, y_*(t), y_*'(t))\,=\,0\,.\end{equation}
It satisfies also the
\textbf{Du Bois-Reymond equation}:
there is a constant $c$ such that, for almost every $t\in I$,
\begin{equation}\label{dure}
L(t, y_*(t), y_*'(t))-y'_*(t)L_v(t, y_*(t), y_*'(t))\,=\,c+\int_a^tL_t(s, y, y')\,ds
\,.\end{equation}
\end{theorem}
\noindent
We shall prove these results in the special case where $L$ and $y_*$ are of class $C^2$.
However, Theorem~\ref{th:euler} does actually hold in a suitable generalized sense with no assumptions on $L$ (in particular without asking it to be differentiable,
see \cite{BM1, BM2}).
\begin{proof}
For the Euler--Lagrange equation,
we use an auxiliary function $\eta$ of class $C^2$
satisfying $\eta(a)=\eta(b)=0$ and we consider
	the energy
	of the perturbed function $y_*+\varepsilon\eta$ where
	 $\varepsilon$ is a small parameter.
	Since $y_*$ is a solution to $\eqref{tag:P}$,
	then the map $\varepsilon\mapsto F(y_*+\varepsilon\eta)$ has
	a local minimum at $\varepsilon=0$, therefore
\[\begin{aligned}
		0\,=\,\frac{d}{d\varepsilon}
		F(y_*+\varepsilon\eta)|_{\varepsilon=0}&\,=\,
\int_a^b
		\frac{d}{d\varepsilon}
L( t, y_*+\varepsilon\eta,
		y'_*+\varepsilon\eta'
)|_{\varepsilon=0}
\,dt\\
&\,=\,
\int_a^b
\eta L_y(t,  y_*, y'_*)\,+\,
\eta' L_v(t,  y_*, y'_*)\,dt\\
&\,=\,
\int_a^b
\eta
\Big(
L_y(t,  y_*, y'_*)\,-\,
		\frac{d}{dt}
L_v(t,  y_*, y'_*)\Big)\,dt\,,
\end{aligned}\]
where we performed an integration by parts in the last step
	and we used the fact that
$\eta(a)=\eta(b)=0$. This equality holds for any suitable function $\eta$,
and this yields the Euler--Lagrange equation~\eqref{eula}.
For the
Du Bois-Reymond equation,
we compute with the help of the chain rule the derivative
\begin{multline}
\dfrac{d}{dt}\big(L(t, y_*(t), y_*'(t))-y'_*(t)L_v(t, y_*(t), y_*'(t))\big)\\
\,=\,L_t(\cdot)+L_y(\cdot)y_*'+L_v(\cdot)y_*''-y_*'\dfrac{d}{dt}L_v(\cdot)-y_*''L_v(\cdot)\qquad (\cdot)=(t, y_*, y_*'))\\
\,=\,L_t(\cdot)+L_y(\cdot)y_*'-y_*'\dfrac{d}{dt}L_v(\cdot)=L_t(t, y_*, y_*'),
\end{multline}
where we have used
the Euler--Lagrange equation~\eqref{eula} in the last step.
\end{proof}
\noindent
The Du Bois-Reymond equation is often useful in order to obtain an explicit expression of the minimizers of \eqref{tag:P}, whenever  they exist.
Let us consider our introductory example~\eqref{ex1}, the problem of
the minimal rotation surface, where $L(y,v)=2\pi y\sqrt{1+v^2}$. The Du Bois-Reymond equation reads
\[y_*\sqrt{1+y_*'^2}-y_*\dfrac{y_*'}{\sqrt{1+y_*'^2}}=c\in\R,\]
which is equivalent to
$y_*^2(1+y_*'^2)\,=\,c^2$.
This equation can be integrated and gives the family of \textbf{ catenaries} (whose name arises from chain), 
\begin{equation}\label{tag:solmin}y_*(t)
\,=\,\smash{\dfrac{1}{\alpha}}\cosh(\alpha t+\beta)\,,\end{equation}
for some constants $\alpha, \beta$ determined by the boundary conditions (see Figure \ref{fig:catenary}). Its rotation around the $x$ axis yields a surface called a \textbf{catenoid}.
A more thorough analysis shows actually that:
\smallskip

\noindent
$\bullet$ If $b-a$ is relatively small with respect to $A$ and $B$ then the solution of the minimal surface problem is of the form given in \eqref{tag:solmin}.
\smallskip

\noindent
$\bullet$ If $b-a$ is sufficiently large then the problem does not have a solution: the infimum in this case is given by $\pi(A^2+B^2)$, which represents the area of the degenerate surface consisting of the two disks of radii $A$ and $B$.
\begin{figure}[!ht]
\begin{center}
\includegraphics[width=0.9\textwidth]{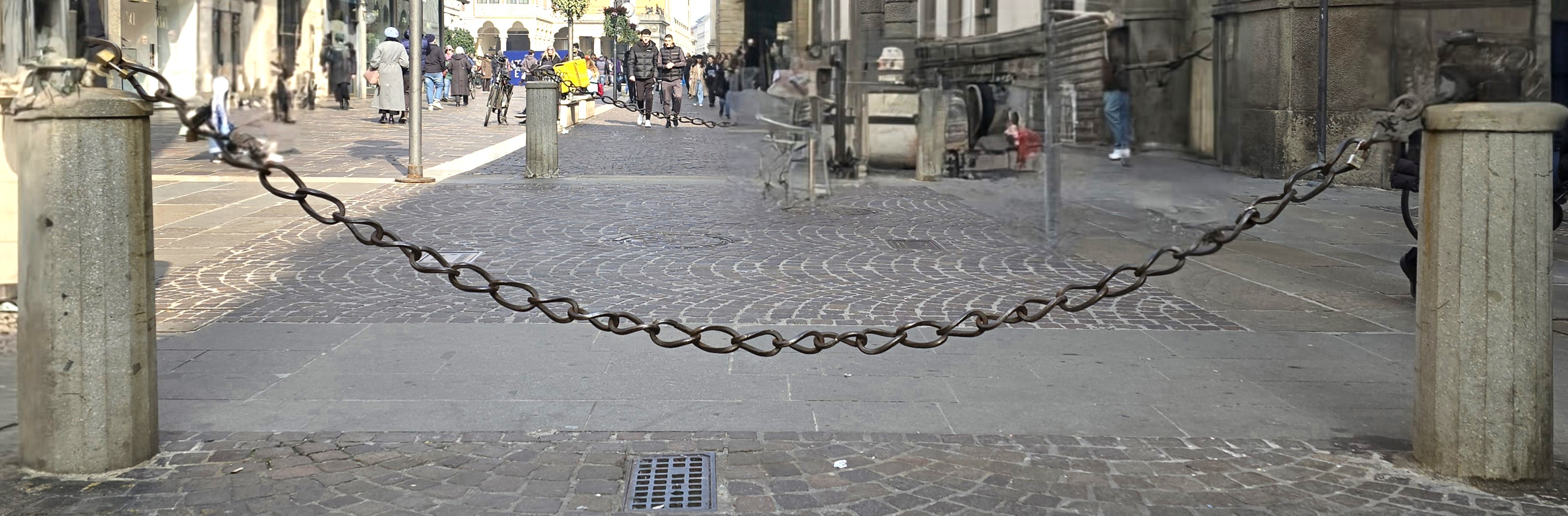}
\caption{{\small A catenary, close to the Bo' (the old building of the University of Padova).}}
\label{fig:catenary}
\end{center}
\vskip-10pt
\end{figure}
\section{The Lavrentiev phenomenon.}\label{sect:cerfm}
Besides the problem of finding a minimizer to $\eqref{tag:P}$,
an important issue is to find the infimum of $F$.
For that goal, one generally relies on the methods of numerical analysis.
However,
these methods work when there is at least a
sequence $(y_n)_{n\geq 1}$  of Lipschitz functions whose values  $F(y_n)$ converge to $\inf F$, i.e., if there is a Lipschitz minimizing sequence  for \eqref{tag:P}.  This occurs, for instance,  in Example \ref{ex:sawtooth} since each function $y_n$ in the  minimizing sequence built there is Lipschitz. Unfortunately this is not always possible.
Although every absolutely continuous function can be approximated
via Lipschitz functions, there are problems that do not admit Lipschitz minimizing sequences.
This truly unpleasant fact was first realized by Lavrentiev in 1926, it involved a very smooth Lagrangian (a polynomial!).
\begin{example}[Mani\`a's example, sequel]\label{ex:mania2}
We come back to
Mani\`a's example \cite{Mania} considered in Example \ref{ex:mania}, which has become a
classical example in the literature on
the Lavrentiev phenomenon.
The functional described in Example~\ref{ex:mania} exhibits
the \textbf{Lavrentiev phenomenon}.
More precisely, there is $\eta>0$ such that, for every Lipschitz function $y$ satisfying the boundary conditions,
\[F(y)\ge \eta>0=F(y_*)=\inf F.\]
This can be shown with the help of
some elementary, though non trivial, computations (see \cite[\S 4.3]{GBH}).
Another unexpected fact appears here: the situation changes drastically if one takes into account just the final end point condition. Indeed it turns out that the
sequence $(y_n)_{n\geq 1}$, where each $y_n$ is obtained by truncating $y_*$ at $1/n$ (Figure \ref{fig: sequenceyn}):
\[y_n(s)\,=\,\begin{cases}{1}/{(n+1)^{1/3}}&\text{ if } s\in [0, \frac1{n+1}],\\
s^{1/3}&\text{ otherwise}, \end{cases}\]
\begin{figure}[!ht]
\begin{center}
\includegraphics[width=0.7\textwidth]{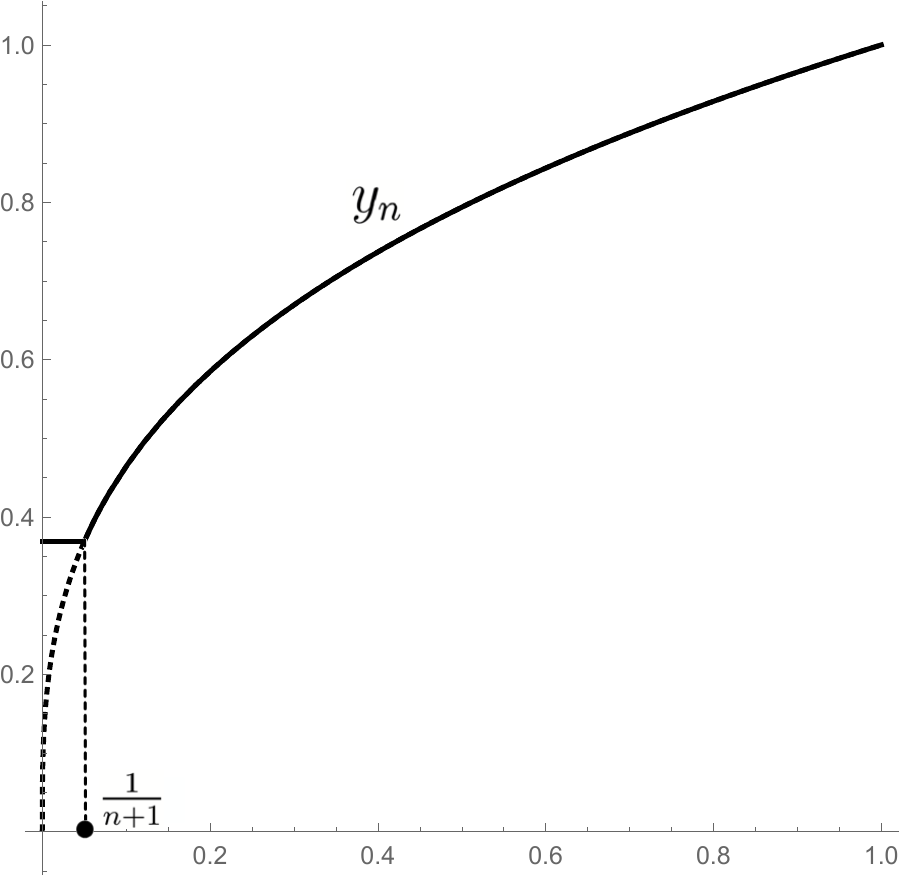}
\caption{{\small The functions $y_n$ for Example~\ref{ex:mania2}
satisfy
$y_n(1)=1$ and $F(y_n)\to F(y_*)$ as $n\to+\infty$.}}
\label{fig: sequenceyn}
\end{center}
\vskip-20pt
\end{figure}
\noindent
is a sequence of Lipschitz functions
satisfying
$y_n(1)=1$ and $F(y_n)\to F(y_*)$ as $n\to+\infty$.
Therefore, unlike the initial problem with
two boundary constraints, no Lavrentiev phenomenon occurs  for the  problem with just the final constraint $y(1)=1$:
\begin{equation}\label{mania2}
\min\, F(y)=\int_0^1(y^3-s)^2(y')^6\,ds:\, y\in \AC([0,1]),\, \, y(1)=1\,.\end{equation}
\end{example}
\begin{example}[A new elementary example]
The next example exhibits the Lavrentiev phenomenon
for the problem
with just one end-point constraint.
The Lagrangian, though, is not as smooth as Mani\`a's  but the method that we use is simpler.
Moreover, the Lagrangian $L$ is autonomous
(whereas
Mani\`a's example is non--autonomous).
This example is a further simplification of  the one
in~\cite[\S 3]{CerfM}.
We consider the one end-point constraint problem
\begin{equation}\label{tag:O}
	\label{E}\min\, F(y)=\,\int_0^1\Big(y'-\frac{1}{2y}\Big)^2\,dt,\qquad  y\in\AC([0,1]),\,\,\, y(0)=0.\end{equation}
To be honest, we are considering here a Lagrangian that, differently from those considered above,  also takes the value $+\infty$:
\[L(y,v)=\begin{cases}\Big(v-\dfrac1{2y}\Big)^2&\text{ if } y\not=0,\\
+\infty&\text{ otherwise}.\end{cases}\]
\end{example}
\noindent
The next theorem
demonstrates that the Lavrentiev phenomenon
occurs for the problem \eqref{tag:O}.
Notice that \eqref{tag:O} has  a minimum, since $F(\sqrt t)=0$.
\begin{theorem}
The energy $F(y)$ is infinite for any
Lipschitz function $y$.
\end{theorem}
\begin{proof}
Let $y$ be a Lipschitz function from $[0,1]$ to $\R$ such that $y(0)=0$.
If $y$ is identically equal to $0$, then its energy $F(y)$ is infinite.
Otherwise, the set of its zeroes,
$$\{\,t\in[0,1]: y(t)=0\,\}\,,$$
is a closed subset of $[0,1]$, its complement is an open
set which can be written as a countable union of disjoint intervals.
Because $y(0)=0$, there exists at least one interval $(a,b)$ included
in $[0,1]$ such that
$y(a)=0$ and $y(t)\neq 0$ for all $t\in [a,b]$.
Since $y$ is Lipschitz, there exists a constant $C$ such that
$$\forall s,t\in[0,1]\qquad|y(t)-y(s)|\,\leq\,C|t-s|\,.$$
Moreover $y$, being Lipschitz, is differentiable almost everywhere and
 its derivative is bounded by the constant $C$.
Let $c\in (a,b)$.
We have
\begin{equation}\label{uui}
F(y)\,\ge\,
\int_{c}^b\left(y'^2-\dfrac{y'}{y}
+\dfrac{1}{4y^2}\right)\,dt
\,\ge\,
-\int_{c}^b\dfrac{y'}{y}\,dt
+
\frac{1}{4}
\int_{c}^b
\dfrac{dt}{y^2}
\,.
\end{equation}
On one hand, we have
\begin{equation}
	\label{oneh}
\int_{c}^b\dfrac{y'(t)}{y(t)}\,dt\,=\,
\ln|y(b)|-\ln|y(c)|\,.
\end{equation}
On the other hand, we have
\begin{equation}
	\label{otheh}
\int_{c}^b
\dfrac{dt}{y(t)^2}\,\geq\,
\frac{1}{C^2}
\int_{c}^b
\Big(\dfrac{y'(t)}{y(t)}\Big)^2\,dt
\,.
\end{equation}
The classical Cauchy--Schwarz inequality yields that
\begin{equation}\label{dtuhi}
\left(\int_c^b
\dfrac{y'(t)}{y(t)}\,dt\right)^2\,\leq\,
(b-c)	\int_{c}^b
\Big(\dfrac{y'(t)}{y(t)}\Big)^2\,dt
\,,
\end{equation}
whence
\begin{equation}\label{tuhi}
	\int_{c}^b
	\Big(\dfrac{y'(t)}{y(t)}\Big)^2\,dt
\,\ge\,
					       \dfrac1{b-c}\Big(
\ln|y(b)|-\ln|y(c)|
\Big)^2\,.
\end{equation}
Substituting the inequalities~\eqref{oneh}, \eqref{otheh} and~\eqref{tuhi}
in inequality~\eqref{uui},
we get
\begin{equation}
	\label{ult}
F(y)\,\ge\,
-\ln|y(b)|+\ln|y(c)|
+
\dfrac1{4C^2(b-c)}\Big(
\ln|y(b)|-\ln|y(c)|
\Big)^2\,.
\end{equation}
Since
$y(a)=0$ and $y$ is continuous at $a$, then $y(c)\to 0$ as $c\to a$, so that
$$\lim_{c\to a}\,\ln|y(c)|\,=\,-\infty\,.$$
Taking the limit in inequality~\eqref{ult}
as $c$ goes to $a$, we conclude that
$F(y)=+\infty$.
\end{proof}
\section{Avoidance of the Lavrentiev phenomenon.}
For\-tu\-na\-te\-ly,
the La\-vren\-tiev phenomenon does not occur
for an autonomous Lagrangian $L(y,v)$
under a very weak condition that is satisfied by any continuous Lagrangian.
\begin{theorem}\label{th:ASCCM} Let $L:\R\times\R\to [0, +\infty)$ be \textbf{bounded on bounded sets}. Then the Lavrentiev phenomenon does not occur for \eqref{tag:P}.
\end{theorem}
\noindent
Theorem~\ref{th:ASCCM} was formulated in \cite{ASC}, proven there for problems with one end-point constraint $y(a)=A$, and proven in \cite{CM5} in the general case with two end-point conditions $y(a)=A, y(b)=B$.
As Example~\ref{ex:mania2} shows, the two problems may behave quite differently with respect to the Lavrentiev phenomenon.
The reader may be puzzled by the fact that the Lagrangian considered in Mani\`a's Example~\eqref{ex:mania}, a polynomial,
is bounded on bounded sets. The fact there, is that the Lagrangian does not depend only on $y$ and $v$ but also on the independent variable $t$.
We prove here a particular case of Theorem~\ref{th:ASCCM}, assuming in addition that $L$ is  of class $C^1$
and that $v\mapsto L(y, v)$ is convex.
Recall that a $C^1$ function $\ell: \R\to \R$  is {convex} if its derivative is non--decreasing,~i.e.,
\[\forall v_1, v_2\in \R\qquad v_1\le v_2\quad\Longrightarrow
\quad\ell'(v_1)\le \ell'(v_2)\,.\]
We state next the main result that we shall prove.
\begin{theorem}
	\label{th:CFM}Assume that $L:\R\times\R\to [0, +\infty)$ is of class $C^1$ and that $v\mapsto L(y,v)$ is \textbf{convex} for every $y\in\R$.
Then the Lavrentiev phenomenon does not occur for \eqref{tag:P}.
\end{theorem}
\noindent
The proof is self--contained, it is a quite simplified version of the one given in \cite{CFM}, where the authors do not assume the $C^1$ regularity of the Lagrangian.
Let us outline the strategy of the proof.
As is quite common in the proofs concerning the Lavrentiev phenomenon, we will actually show more than what is claimed.
More precisely, we will show that, for any $y$ in $\AC(I)$
satisfying the boundary conditions
and such that $F(y)<+\infty$,
there is a sequence of Lipschitz functions $(y_k)_{k\geq 1}$
satisfying the boundary conditions and such that
$F(y_k)\le F(y)+1/{k}$ for any $k\geq 1$.
Thus, given any minimizing sequence for \eqref{tag:P} in $\AC(I)$,
we obtain a minimizing sequence for \eqref{tag:P} consisting
of Lipschitz functions.
Indeed, let $(y_m)_{m\geq 1}$ be any minimizing sequence for \eqref{tag:P} and, for each $m\geq 1$, choose a Lipschitz function $\overline y_m$ satisfying the boundary conditions and such that $F(\overline y_m)\le F(y_m)+1/{m}$. Then, for each $m\geq 1$,
\[\inf F\le F(\overline y_m)\le F(y_m)+\dfrac1{m}
\to \inf F\quad\text{as}\quad m\to +\infty\,,\]
thereby
proving that the Lavrentiev phenomenon does not occur.
Now,
given $y$ in $\AC(I)$
satisfying the boundary conditions
such that $F(y)<+\infty$, how can we build the sequence
$(y_k)_{k\geq 1}$
of Lipschitz functions approximating $y$?

The trick consists in making a judicious change of time. We build
a sequence of bijective functions $\varphi_k: I\to I$ and we set
$y_k(t)=y(\varphi_k^{-1}(t))$ (steps 1 and 2 of the proof).
The change of time $\varphi_k$ is designed
so that the resulting function $y_k$ is Lipschitz. So, on
the set $S_k$ where the derivative $y'$ has a modulus larger than $k$, we slow down the time by a factor $k/|y'|$.
This would do
the job if there was only one constraint at the beginning of the interval. Now, in order to ensure that $y_k$ still satisfies
the constraint at the end of the interval, we have to accelerate
the time on another subset $A_k$ of $I$.
The total acceleration has to be finely tuned so that $y_k(b)=B$ (step 3 of the proof),
and simultaneously
the acceleration factor
must remain bounded so that the
function $y_k$ is Lipschitz
(step 4 of the proof).
In addition,
the total modification induced by the time change has to
be controlled so that $F(y_k)$ stays close to $F(y)$
(steps 5, 6 and 7 of the proof).
The heart of the proof, inspired by \cite{CM8}, consists in the construction of
the two sets $A_k$ and $S_k$ satisfying these requirements.
\subsection{Proof of Theorem \ref{th:CFM}.}
Fix $y$ in $\AC(I)$ such that $F(y)<+\infty$.
For $E$ a measurable set in $\R$, we denote by $|E|$ its Lebesgue measure.
We divide the proof into seven steps.\\
Step 1: {\em The set $S_k$ and choice of $\lambda$.}
For every integer $k\ge 1$, we set
\[S_k\,=\,\{\,t\in I:\, |y'(t)|\ge k\,\}\,.\]
Since $|y'|\ge k$ on $S_k$, we have
\begin{equation}\label{fili}|S_k|
\,\le\, \dfrac1k\int_{S_k}|y'(t)|\,dt\,\le\, \dfrac1k\int_I|y'(t)|\,dt\,\to\, 0\quad
\text{as}\quad k\to +\infty\,.\end{equation}
We also choose an arbitrary value $\lambda$ such that
the set
$\Omega_\lambda\,=\,
	\big\{\,t\in I:\, |y'(t)|\le \lambda\,\big\}$
has positive Lebesgue measure.
This is possible since
$I\,=\,\bigcup_{n=0}^{\infty}\Omega_n$:
if all the sets in this union were negligible,
then $I$ itself would be negligible.
\smallskip

\noindent
Step 2: {\em The change of variable $\varphi_k$.} We define the function $v_k:I\to \R$ by setting
    \[v_k(t)=\begin{cases}{|y'(t)|}/{k}&\text{ if }t\in S_k,\\
    1/2&\text{ if } t\in A_k,\\
    1&\text{ otherwise,}\end{cases}\]
    where $A_k\subseteq {I\setminus }S_k$ is a suitable subset of
    $I$ where $|y'|\le \lambda$,
    which  will be defined below.
    We define
a function $\varphi_k$ on $I$ by
\[\forall t\in I\qquad \varphi_k(t)\,=\,a+\int_a^tv_k(\tau)\,d\tau.\]
The function $\varphi_k$ belongs to $\AC(I)$ and
$\varphi_k'=v_k$ almost everywhere.
Notice that $\varphi_k(a)=a$ and $\varphi_k$ is strictly increasing
since, for $t_1<t_2$ in $I$,
\[\varphi_k(t_2)-\varphi_k(t_1)\,=\,\int_{t_1}^{t_2}v_k(\tau)\,d\tau\,>\,0\,.\]

\noindent
Step 3: {\em The set $A_k$.} We wish that the image of $\varphi_k$ is $I$. This happens  if
$\varphi_k(b)=b$, i.e.,
\[\begin{aligned}b-a\,=\,|I|\,=\,\int_a^bv_k\,dt
&\,=\,\int_{S_k}\dfrac{|y'|}{k}\,dt+
\int_{A_k}
\dfrac12
\,dt+\int_{I\setminus (S_k\cup A_k)}1\,dt\\
&\,=\,\int_{S_k}\dfrac{|y'|}{k}\,dt+\dfrac{|A_k|}{2}+|I|-|S_k|-|A_k|\,.
\end{aligned}\]
Taking into account that $|S_k|=\displaystyle\int_{S_k}1\,dt$, the previous condition can be rewritten as
\begin{equation}\label{tag:defSigma}
2\displaystyle\int_{S_k}\Big(\dfrac{|y'|}{k}-1\Big)\,dt
\,=\,{|A_k|}\,.\end{equation}
Notice that, since $|y'|\ge k$ on $S_k$, then, using~\eqref{fili},
\begin{equation}\label{fiji}
	0\,\le\,
	\int_{S_k}\Big(\dfrac{|y'|}{k}-1\Big)\,dt
\,\le\, \dfrac1k\int_I|y'|\,dt-|S_k|\to 0\quad\text{as}\quad
k\to +\infty\,.\end{equation}
We will take $A_k$ to be a subset of the set
$\Omega_\lambda$ introduced in step 1.
For $k>\lambda$, the sets $\Omega_\lambda$ and $S_k$ are disjoint, thus any
subset
$A_k$
of $\Omega_\lambda$ satisfies
\begin{equation}
	\label{firsttwo}
	A_k\subseteq I\setminus S_k\,,\quad
	|y'|\le \lambda\text{ on $A_k$}\,.
\end{equation}
Thanks to the limit~\eqref{fiji}, we can find
 $k>\lambda$ large enough so that
 \[
	 |\Omega_\lambda|\,>\,2
	 \displaystyle\int_{S_k}\Big(\dfrac{|y'|}{k}-1\Big)\,dt
 \]
and we choose for $A_k$
a subset of $\Omega_\lambda$
whose Lebesgue measure satisfies~\eqref{tag:defSigma}.
\goodbreak
%
    \smallskip

\noindent
Step 4: {\em The sequence $(y_k)_{k\geq 1}$ of Lipschitz functions.}
At this stage,
we assume that $k$ is large enough
so that there exists a set $A_k$ satisfying the conditions~\eqref{tag:defSigma} and~\eqref{firsttwo}.
The associated function $\varphi_k:I\to I$ is
bijective. We denote by $\psi_k$ its inverse
and we set
    \[\forall s\in I\qquad y_k(s)=y(\psi_k(s)).\]
    Clearly $y_k$ satisfies the required boundary conditions:
 \[y_k(a)=y(\psi_k(a))=y(a)=A,\quad y_k(b)=y(\psi_k(b))=y(b)=B.\]
 Since $v_k'\ge1/2$ then $\psi_k$ is Lipschitz; it follows   that $y_k$ is absolutely continuous (see \cite[IX,\S3, Theorem 5]{Natanson}).
 Moreover $y_k$ is Lipschitz: indeed, from the chain  rule \cite{Serrin},
 \[\forall s\in I\qquad y_k'(s)
 \,=\,\dfrac{y'(\psi_k(s))}{\varphi_k'(\psi_k(s))}
 \,=\,
 \begin{cases}k\dfrac{y'(\psi_k(s))}{|y'(\psi_k(s))|}&\text{ if }s\in \varphi_k(S_k)\,,\\
    2y'(\psi_k(s))&\text{ if } s\in \varphi_k(A_k)\,,\phantom{\frac{y'}{y'}}\\
    y'(\psi_k(s))&\text{ otherwise\,.}\phantom{\frac{y'}{y'}}\end{cases}\]
Since $|y'|\le k$ outside of $S_k$, we have $|y_k'|\le 2k$ on $I$.
\smallskip

\noindent
Step 5: {\em The value of $F(y_k)$.}
 The change of variable  $t=\psi_k(s)\Leftrightarrow s=\varphi_k(t)$ gives
 \begin{equation}
 \begin{aligned}F(y_k)&=\int_a^bL(y_k(s), y_k'(s))\,ds=\int_a^bL\Big(y(\psi_k(s)), \dfrac{y'(\psi_k(s))}{\varphi_k'(\psi_k(s))}
 \Big)\,ds\\
 &=\int_a^bL\Big(y(t), \dfrac{y'(t)}{\varphi_k'(t)}\Big)\varphi_k'(t)\,dt\,.
 \end{aligned}
 \end{equation}
Since $\varphi_k'=1$ on $I\setminus (S_k\cup A_k)$
and $\varphi_k'=1/2$ on $A_k$, we decompose $F(y_k)$ as
 \begin{equation}
	 \label{tag:integralF}
		 F(y_k)\,=\,\int_{I\setminus (S_k\cup A_k)}L(y(t),y'(t))\,dt+\mathcal I_{A_k}+\mathcal I_{S_k}\,,
	 \end{equation}
 where we set
 \[\mathcal I_{A_k}=\dfrac12\int_{A_k}L(y, 2y')\,dt\,,\quad \mathcal I_{S_k}=\int_{S_k}L\Big(y, \dfrac{y'}{\varphi_k'}
 \Big)\varphi_k'\,dt\,.
 \]
Recall that our aim is to prove that, for $k$ large enough,
$F(y_k)\le F(y)+1/{k}$.
Since $L$ is non--negative, we have obviously from~\eqref{tag:integralF}
that
\begin{equation}
	\label{ulti}
	F(y_k)\,\le\,
	F(y)+\mathcal I_{A_k}+\mathcal I_{S_k}\,.
\end{equation}
Our final goal is to estimate the two terms
 $\mathcal I_{A_k}$, $\mathcal I_{S_k}$.
    \smallskip

\noindent
Step 6:
{\em Estimate of $\mathcal I_{A_k}$.}
This is the easiest of the two terms to estimate. Indeed, recall that $|y'|\le \lambda$ on $A_k$. Now,
the range $y(I)$ of $y$ is a bounded interval and
$L$, being a continuous function, is bounded by a constant $m$ on $y(I)\times [-2\lambda, 2\lambda]$. It follows from~\eqref{fiji} that
    \begin{equation}\label{tag:ISIGMA}\mathcal I_{A_k}\,\le\, \dfrac12\int_{A_k}m\,dt=\dfrac{m}2|A_k|\,\to\, 0
    \quad\text{as}\quad k\to +\infty\,.\end{equation}
Step 7:
{\em Estimate of $\mathcal I_{S_k}$.}
We will apply the following Lemma~\ref{prop:derivative}
to the function
 $\ell(v)=L(y,v)$ for a fixed value of $y\in\R$.
 \begin{lemma}\label{prop:derivative} Let $\ell$ be a convex and $C^1$ function from $\R$ to $\R$
and let $v\in \R$. The function
$\phi_v$ defined by
$$\forall \mu>0\qquad \phi_v(\mu)\,=\,\smash{\ell\Big(\frac{v}{\mu} \Big)}\mu\,$$
is convex
on $(0, +\infty)$ and its derivative is given by
$\phi_v'(\mu)=P\big({v}/{\mu}\big)$, where
\begin{equation}
\label{dfp}
\forall w\in\R\qquad P(w)\,=\,\ell(w)-w\ell'(w)\,.
\end{equation}
%
This function $P$ is non-decreasing on $(-\infty, 0)$ and non-increasing on $(0, +\infty)$.
\end{lemma}
\noindent
There is a natural geometric interpretation of the function $P$. Indeed,
for $w\in\R$, $P(w)$ is the ordinate of the intersection of  the vertical axis and the tangent line to the graph of $\ell$ at $w$ (see Figure~\ref{fig:tangent}).
\begin{figure}[!ht]
\begin{center}
\includegraphics[width=0.7\textwidth]{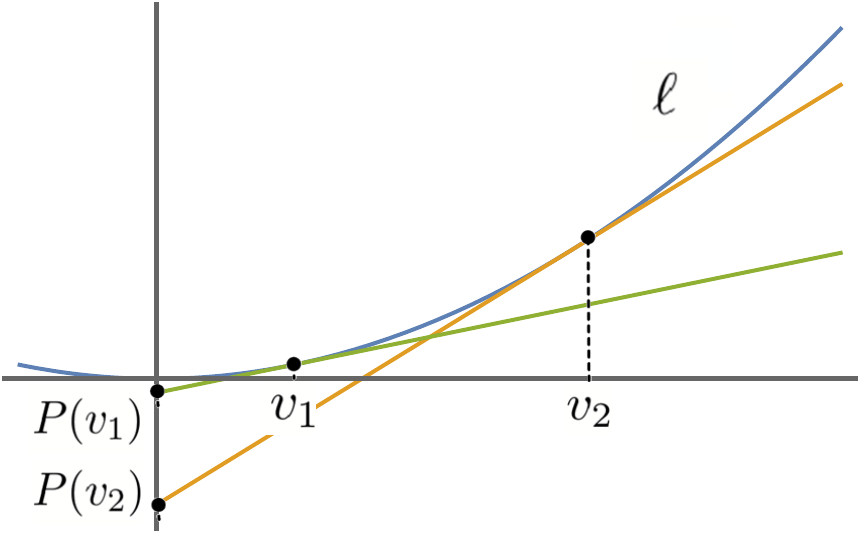}
\caption{{\small Monotonicity of $P$ and the graphical interpretation of $P(v)=\ell(v)-v\ell'(v)$.}}
\label{fig:tangent}
\end{center}
\vskip-20pt
\end{figure}
\begin{proof}[Proof of Lemma~\ref{prop:derivative}]
Let $v\in\R$ be fixed.
The derivative of $\phi_v$ is given by
\begin{equation}
\label{derph}
\forall  \mu>0\qquad
\phi_v'(\mu)=\ell\Big(\dfrac v{\mu}\Big)-\dfrac{v}{\mu}\ell'\Big(\dfrac v{\mu}\Big)=P\Big(\dfrac v{\mu}\Big)\,,\end{equation}
where $P$ is the function defined in~\eqref{dfp}.
Let $v_1<v_2$, we have
\[
	P(v_2)-P(v_1)\,=\,
	\ell(v_2)-v_2\ell'(v_2)-
	\ell(v_1)+v_1\ell'(v_1)\,.
\]
By the mean value theorem, there exists $c\in (v_1,v_2)$ such that
	$\ell(v_2)-\ell(v_1)=\ell'(c)(v_2-v_1)$, whence
	\begin{equation}
		\label{sipl}
		P(v_2)-P(v_1)\,=\,
		v_2(\ell'(c)-\ell'(v_2))+
		v_1(\ell'(v_1)-\ell'(c))
		\,.
	\end{equation}
Since the function $\ell$ is convex, its derivative $\ell'$ is non--decreasing,
and the terms in parenthesis in~\eqref{sipl} are non--positive. We conclude
that
$P$ is non-decreasing on $(-\infty, 0)$ and non-increasing on $(0, +\infty)$.
The formula~\eqref{derph} shows then that the derivative of
$\phi_v$ is non-decreasing on
$(0, +\infty)$, hence
$\phi_v$ is convex on
$(0, +\infty)$.
\end{proof}\noindent
 Let's continue with step 7 of the proof.
 Fix $t\in S_k$ and consider the auxiliary function
    $\Phi_{k,t}$
    defined by
    \[\Phi_{k,t}(\mu)=L\Big(y(t), \frac{y'(t)/\varphi_k'(t)}{\mu}\Big)\mu\,.\]
    \goodbreak
\noindent
By hypothesis, the map
$v\mapsto L(y, v)$ is convex, hence
Lemma~\ref{prop:derivative}  applied to
the function
 $\ell(v)=L(y,v)$ for a fixed value of $y\in\R$
implies that
$\Phi_{k,t}$ is convex on $(0,+\infty)$.
Using the classical fact that the graph of a convex function lies
above its tangent line at any point, we obtain
\begin{equation}\label{mult}
\Phi_{k,t}\Big(\dfrac1{\varphi_k'(t)}\Big)-\Phi_{k,t}(1)\ge \Phi_{k,t}'(1)\Big( \dfrac1{\varphi_k'(t)}-1 \Big)\,.\end{equation}
The derivative of
    $\Phi_{k,t}$ is also computed in
Lemma~\ref{prop:derivative}. We have
\begin{equation}\label{dphi}
\Phi_{k,t}'(1)=P\Big(\dfrac{y'(t)}{\varphi_k'(t)}\Big)\end{equation}
where
$P(v)=L(y(t),v)-vL_v(y(t), v).$
Noticing that
    \[\Phi_{k,t}(1)=L\Big(y(t), \frac{y'(t)}{\varphi_k'(t)}\Big),\quad \Phi_{k,t}\Big(\dfrac1{\varphi_k'(t)}\Big)=L(y(t), y'(t))\dfrac1{\varphi_k'(t)}\,,\]
    multiplying both sides of the inequality~\eqref{mult} by $\varphi_k'(t)$, and using~\eqref{dphi}, we get
\begin{equation}\label{tag:ineqfin}L\Big(y(t), \frac{y'(t)}{\varphi_k'(t)}\Big)\varphi_k'(t)\le L(y(t), y'(t))+
P\Big(\dfrac{y'(t)}{\varphi_k'(t)}\Big)
    (\varphi_k'(t)-1).\end{equation}
Since $t\in S_k$, then $\varphi'_k(t)=|y'(t)|/k$ 
and therefore
\[
P\Big(\dfrac{y'(t)}{\varphi_k'(t)}\Big)
\,=\,
\begin{cases}P(k)&\text{ if }y'(t)>0\\P(-k)&\text{ if }y'(t)<0\end{cases}\,.\]
It was also shown in
Lemma~\ref{prop:derivative} that
$P$ is non-decreasing on $(-\infty, 0)$ and non-increasing on $(0, +\infty)$,
thus
\begin{equation}\begin{aligned}P(-k)\,\le\, P(-1)&\,=\,L(y(t),-1)+L_v(y(t), -1)\,,\\
P(k)\,\le\, P(1)&\,=\,L(y(t),1)-L_v(y(t), 1)\,.\end{aligned}\end{equation}
Since $t\mapsto L(y(t), \pm 1), t\mapsto L_v(y(t), \pm 1)$ are bounded on $I$, there exists $M\ge 0$ (not depending on $k, t$) satisfying
$
P\big({y'(t)}/{\varphi_k'(t)}\big)
\,\le\, M$ on $S_k$.
Since
$\varphi_k'(t)={|y'(t)|}/k\ge 1$
for $t\in S_k$,
it follows from \eqref{tag:ineqfin} that, on $S_k$,
\[L\Big(y(t), \frac{y'(t)}{\varphi_k'(t)}\Big)\varphi_k'(t)\le L(y(t), y'(t))+M \Big(\dfrac{|y'(t)|}k-1\Big).\]
By integrating {both terms of the inequality }on $S_k$, we get
\begin{equation}\label{tag:IS}0\le \mathcal I_{S_k}\le \int_{S_k}L(y(t), y'(t))\,dt+M
	\int_{S_k}\Big(\dfrac{|y'(t)|}{k}-1\Big)\,dt\,.
\end{equation}
The Lebesgue measure of $S_k$ goes to
$0$ as $k$ goes to $\infty$
by~\eqref{fili},
and so does
the first integral by the dominated convergence theorem.
The second integral goes to $0$ as $k$ goes to $\infty$ by~\eqref{fiji}.
Thus it follows from \eqref{ulti},  \eqref{tag:ISIGMA} and \eqref{tag:IS}  that, for $k$ large enough,
$F(y_k)\le F(y)+1/{k}$, and this concludes the proof of
the theorem.$\qed$
\section*{Acknowledgments.}
The authors would like to express their gratitude to the reviewers for their thorough and careful review of the manuscript, as well as for their invaluable comments.

\bibliographystyle{amsplain}
\bibliography{Monthly2}





\end{document}